\documentclass{amsart}

\newtheorem{theorem}{Theorem}[section]
\newtheorem{lemma}[theorem]{Lemma}
\newtheorem{proposition}[theorem]{Proposition}
\newtheorem{corollary}[theorem]{Corollary}

\theoremstyle{definition}
\newtheorem{definition}[theorem]{Definition}

\newtheorem{problem}[theorem]{Problem}

\theoremstyle{remark}
\newtheorem{remark}[theorem]{Remark}

\usepackage{amssymb,amsmath}

\numberwithin{equation}{section}

\def\R{{\mathbb R}}

\def\Z{{\mathbb Z}}

\def\C{{\mathbb C}}

\def\cA{{\mathcal A}}

\def\cL{{\mathcal L}}

\def\Re{{\rm Re}\,}
\def\Im{{\rm Im}\,}

\def\TcT{({\mathbf T} (t))_{t\geq 0}}

\def\TcT0{({\mathbf T}_0 (t))_{t\geq 0}}

\def\L1{L^1 (\R_+ )}
\def\mpd{M_*^{{~}*}}
\begin{document}

\title[Operators $L^1 (\R_+ )\to X$ and norm continuity of semigroups]{Operators $L^1 (\R_+ )\to X$ and the norm continuity problem for semigroups}

% \author{Charles J. K. Batty}
% \address{St. John's College, Oxford OX 13 JP, Great Britain}
% \curraddr{ }
% \email{charles.batty@sjc.ox.ac.uk}

\author{Ralph Chill}
\address{Universit\'e Paul Verlaine - Metz, Laboratoire de Math\'ematiques et Applications de Metz - CNRS, UMR 7122, B\^at. A, Ile du Saulcy, 57045 Metz Cedex 1, France}
\email{chill@univ-metz.fr}

%    Information for third author
\author{Yuri Tomilov}
\address{Faculty of Mathematics and Computer Science, Nicolas Copernicus University, ul. Chopina 12/18, 87-100 Torun, Poland}
\email{tomilov@mat.uni.torun.pl}
\thanks{Supported by the Marie Curie ''Transfer of Knowledge'' programme, project ''TODEQ''. This research was started during a Research in Pair stay at the Mathematisches Forschungsinstitut Oberwolfach. Supports are gratefully acknowledged.}

%    General info

\date{May 22, 2008}

%\dedicatory{This paper is dedicated to our authors.}


\begin{abstract}
We present a new method for constructing
$C_0$-semigroups for which properties of the resolvent of the generator and 
continuity properties of the semigroup in the operator-norm topology are controlled simultaneously.
It allows us to show that a) there exists a $C_0$-semigroup which is continuous in the operator-norm topology for no $t \in [0,1]$  such that the resolvent of its generator has a logarithmic decay at infinity along vertical lines;
b)  there exists a $C_0$-semigroup which is continuous in the operator-norm topology for no 
$t \in \mathbb R_+$ such that the resolvent of its generator has a decay along vertical lines arbitrarily close to a logarithmic one. These examples rule out any possibility of characterizing norm-continuity of semigroups on arbitrary Banach spaces in terms of resolvent-norm decay on vertical lines.
\end{abstract}

\maketitle

\section{Introduction}
The study of continuity properties of $C_0$-semigroups $(T(t))_{t\geq 0}$ on a Banach space $X$ in the uniform operator topology of $\cL (X)$ (norm-continuity) has been initiated in \cite{HiPh57} and attracted considerable attention over the last decades; see in particular \cite{Ba07}, \cite{Bl01}, \cite{BlMa96}, \cite{ElMEn94}, \cite{ElMEn96}, \cite{GoWe99}, \cite{Il07},  \cite{Ma08}, \cite{Ya95}, \cite{Yo92}.

The classes of immediately norm-continuous semigroups, of eventually norm-continuous semigroups,
and of asymptotically norm-continuous semigroups (or, equivalently, semigroups norm continuous at infinity)
emerged and were studied in depth during this period.
The interest in these classes comes mainly from the fact that a condition of norm continuity of a semigroup implies a variant of the spectral mapping theorem, and thus asymptotic properties of a semigroup are essentially determined by the spectrum of the generator.

One of the main issues in the study of norm-continuity is to characterize these classes in terms of the resolvent of the semigroup generator (or in other a priori terms).
In particular, the so called norm-continuity problem for $C_0$-semigroups attributed to A. Pazy was a focus for relevant research during the last two decades.

Given a $C_0$-semigroup $(T(t))_{t\geq 0}$ on a Banach space $X$, with generator $A$, the problem is to determine whether the resolvent decay condition 
\begin{equation} \label{res-cond}
\lim_{|\beta |\to\infty} \| R(\omega +i\beta , A)\| = 0 \text{ for some } \omega\in\R 
\end{equation}
implies that the semigroup is  immediately norm-continuous, that is, norm-continuous for $t>0$. The decay condition \eqref{res-cond} is certainly necessary for immediate norm-continuity, by the fact that the resolvent of the generator is the Laplace transform of the semigroup, and by a simple application of the Lemma of Riemann-Lebesgue. Hence, the question is whether condition \eqref{res-cond} characterizes immediate norm-continuity. \\

The resolvent decay condition \eqref{res-cond} does characterize immediate norm continuity if the underlying Banach space is a Hilbert space \cite{Yo92}, \cite{ElMEn94}, \cite{Ya95}, \cite[Theorem 3.13.2]{ABHN01}, or if it is an $L^p$ space and the semigroup is positive, \cite{GoWe99}. Only very recently, T. Matrai \cite{Ma08} constructed a counterexample showing that the answer to the norm continuity problem is negative in general. The generator in his example is an infinite direct sum of Jordan blocks on finite dimensional spaces. The infinite sum is equipped with an appropriate norm and the resulting Banach space is reflexive. This kind of counterexample going back to \cite{Za75} has been used in the spectral theory of semigroups to show the failure of the spectral mapping theorems or  certain relationships between semigroup growth bounds, see for example \cite{ABHN01}, \cite{EnNa99}.\\

We point out that the resolvent decay condition \eqref{res-cond} implies that the resolvent exists and is uniformly bounded in a domain of the form
$$
\Sigma_\varphi := \{ \lambda\in\C : {\rm Re}\, \lambda > - \varphi (|{\rm Im}\, \lambda |) \} ,
$$
where $\varphi\in C(\R_+ )$ satisfies $\lim_{\beta\to\infty} \varphi (\beta ) = \infty$. It is known that the existence of the resolvent and its uniform boundedness in such a domain can imply regularity properties of the semigroup if the function $\varphi$ is growing sufficiently fast: we recall corresponding results for analytic, immediately differentiable and eventually differentiable semigroups, \cite[Theorem 3.7.11]{ABHN01}, \cite[Theorems 4.7, 5.2]{Pa83}. It follows from the proofs of these results (which use the complex inversion formula for Laplace transforms) that there are similar results in the more general context of Laplace transforms of vector-valued functions; see, for example, \cite[Theorem 2.6.1]{ABHN01}, \cite[I.4.7, II.7.4]{Do50}, \cite{Sv79}, \cite{Re87}. One could therefore think of the following Laplace transform version of the norm-continuity problem: if the Laplace transform of a bounded scalar (or vector-valued) function extends analytically to a bounded function in some domain $\Sigma_\varphi$, where $\varphi\in C(\R_+ )$ satisfies $\lim_{\beta\to\infty} \varphi (\beta ) = \infty$, is the function immediately or eventually continuous? It is relatively easy to give counterexamples to this Laplace transform version of the norm-continuity problem. It follows from the main result in this article (Theorem \ref{main}) that every counterexample to the norm-continuity problem for scalar functions yields a counterexample to the the norm-continuity problem for semigroups.\\

As indicated in the title of this article, we approach the problem of norm-continuity via Banach algebra homomorphisms $\L1 \to \cA$. 
The connection between semigroups and such homomorphisms is well-known.
We recall that to every bounded $C_0$-semigroup $(T(t))_{t\geq 0}$ on a Banach space $X$ one can associate an algebra homomorphism $T:\L1 \to \cL (X)$ given by
\begin{equation}\label{hom-def}
Tg = \int_0^\infty T(t) g(t) \; dt , \quad g\in\L1 \quad \text{(integral in the strong sense)}.
\end{equation}
Conversely, every algebra homomorphism $\L1 \to \cA$ is, after passing to an equivalent homomorphism, of this form; cf. Lemma \ref{equivalenthomomorphism} below. 

It is therefore natural to ask how regularity properties of the semigroup or the resolvent of its generator are encoded in the corresponding algebra homomorphism or its adjoint. We will discuss some of the connections in the first part of this article, partly in the context of general operators $\L1\to X$. Then, given a function $f\in L^\infty (\R_+ )$ such that its Laplace transform extends to a bounded analytic function on some domain $\Sigma_\varphi$, we will show how to construct an algebra homomorphism $T: \L1 \to \cL (X)$ which is represented (in the strong sense) by a $C_0$-semigroup such that $f\in {\rm range}\, T^*$ and such that the resolvent of the generator satisfies a precise decay estimate. In fact, the space $X$ will be continuously embedded into $L^\infty (\R_+ )$ and left-shift invariant, and the operator $T$ will be represented by the left-shift semigroup on $X$. In this way, we will be able to show that the norm-continuity problem has a negative solution and at the same time we will be able to estimate the resolvent decay along vertical lines. It turns out that the decay $\| R(\omega +i\beta ,A)\| = O(1/\log |\beta |)$ implies eventual differentiability but not immediate norm-continuity, and that any slower decay does not imply eventual norm-continuity.

The paper is organized as follows. In the second  section, we remind some basic properties
and definitions from the theory of operators $L^1\to X$ needed in the sequel,
and introduce the notion of a Riemann-Lebesgue operator. In the third section,
we set up a framework of homomorphisms $L^1\to \cA$ and establish the relation
to the norm continuity problem for semigroups. The main, fourth section, is devoted to the construction of Riemann-Lebesgue homomorphisms. Finally, in the fifth section, we apply the main result from the fourth section to give counterexamples to the norm-continuity problem. 

\section{Operators $\L1 \to X$}

Operators $L^1\to X$ and their representations is a classical subject  of both  operator theory and  geometric theory of Banach spaces. For a more or less complete account of basic properties of these operators one may consult \cite{DiUh77}, and a selection of more recent advances pertinent to our studies include \cite{Bo80II}, \cite{Bo80}, \cite{CrGo84}, \cite{GhTa84}, \cite{Gi90}, \cite{GiJo97}, \cite{Ho89}, \cite{Iv07}, \cite{KPRU89}.

The following representation of operators $\L1 \to X$ by vector-valued Lip\-schitz continuous functions on $\R_+$ will be used in the sequel. We denote by ${\rm Lip}_0 (\R_+ ;X)$ the Banach space of all Lipschitz continuous functions $F:\R_+ \to X$ satisfying $F(0)=0$. Then, for every $F\in {\rm Lip}_0 (\R_+ ;X)$ the operator $T_F: \L1 \to X$ given by the Stieltjes integral
\begin{equation} \label{TF}
T_F g := \int_0^\infty g(t) \; dF(t) , \quad g\in\L1 ,
\end{equation}
is well defined and bounded, and it turns out that {\em every} bounded operator $T:\L1 \to X$ is of this form. In fact, by the Riesz-Stieltjes representation theorem \cite[Theorem 2.1.1]{ABHN01}, the operator $F \to T_F$ is an (isometric) isomorphism from ${\rm Lip}_0 (\R_+ ;X)$ onto $\cL (\L1 ,X)$. \\ 

There are several analytic properties of operators $L^1 \to X$ which have been defined and studied in the literature. Among them, we will recall Riesz representability and the (local) Dunford-Pettis property, and we introduce the Riemann-Lebesgue property. The first and the last will be relevant for this article while the (local) Dunford-Pettis property is mentioned for reasons of comparison. \\

Throughout the following, for every $\lambda\in\C$ and every $t\in\R_+$, we define $e_{\lambda} (t) := e^{-\lambda t}$. If $\lambda$ belongs to the open right half-plane $\C_+$, then $e_\lambda \in \L1$.  

Recall that an operator between two Banach spaces is called {\em Dunford-Pettis} or {\em completely continuous} if it maps weakly convergent sequences into norm convergent sequences.

\begin{definition} \label{tproperties}
Let $T:\L1\to X$ be a bounded operator. 
\begin{itemize}
\item[(a)] We call $T$ {\em Riesz representable}, or simply {\em representable}, if there exists a function $f\in L^\infty (\R_+ ;X)$ such that 
\begin{equation} \label{rep1}
T g = \int_{\R_+} g(s) f(s) \; ds \text{ for every } g\in \L1 .
\end{equation}
\item[(b)] We call $T$ {\em locally Dunford-Pettis} if for every measurable $K\subset\R_+$ of finite measure the restriction of $T$ to $L^1 (K)$ is Dunford-Pettis. 
\item[(c)] We call $T$ {\em Riemann-Lebesgue} if $\lim_{|\beta|\to\infty} \| T (e_{i\beta} g)\| = 0$ for every $g\in \L1$. 
\end{itemize}
\end{definition}

The definitions of representable and local Dunford-Pettis operators clearly make sense on general $L^1$ spaces. For some operator theoretical questions it may be more natural to consider operators on $L^1 (0,1)$ or a similar $L^1$ space. In the context of (bounded) $C_0$-semigroups, the space $\L1$ is appropriate.

We point out that if $F\in {\rm Lip}_0 (\R_+ ;X)$ and if $T = T_F:\L1 \to X$ is representable by some function $f\in L^\infty (\R_+ ;X)$, then necessarily $F(t) = \int_0^t f(s) \; ds$. In fact, $T=T_F$ is representable if and only if the function $F$ admits a Radon-Nikodym derivative in $L^\infty (\R_+ ;X)$. 

\begin{proposition} \label{implications}
Let $T:\L1\to X$ be a bounded operator. The following implications are true:
\begin{eqnarray*}
& T \text{ is weakly compact} & \\
& \Downarrow & \\
& T \text{ is representable} & \\
& \Downarrow & \\
& T \text{ is locally Dunford-Pettis} & \\
& \Downarrow & \\
& T \text{ is Riemann-Lebesgue.} & 
\end{eqnarray*}
\end{proposition}

\begin{proof}
The first implication follows from \cite[Theorem 12, p.75]{DiUh77}, while the second implication is a consequence of \cite[Lemma 11, p.74, and Theorem 15, p.76]{DiUh77}. Note that these results only deal with finite measure spaces whence the necessity to consider local Dunford-Pettis operators. 

In order to prove the last implication, one has to remark that for general $T:\L1 \to X$ the space
\begin{equation} \label{rlclosed}
E := \{ g\in \L1 : \lim_{\beta\to\infty} \| T(e_{i\beta} g)\| =0 \} \text{ is closed in } \L1 .
\end{equation}
Next, by the Lemma of Riemann-Lebesgue, for every $g\in L^1 (\R_+ )$ one has $w-\lim_{\beta\to\infty} e_{i\beta} g = 0$ in $\L1$ and $L^1 (K )$, where $K$ is any compact subset of $\R_+$. Hence, if $T$ is locally Dunford-Pettis, then the space $E$ contains all compactly supported functions in $\L1$, and since this space is dense in $\L1$, the operator $T$ must be Riemann-Lebesgue. 
\end{proof}

The properties from Definition \ref{tproperties} have also been defined for Banach spaces instead of single operators. For example, a Banach space $X$ has the {\em Radon-Nikodym property} if every operator $T:\L1 \to X$ is representable, \cite{DiUh77}, or, equivalently, if every function in ${\rm Lip}_0 (\R_+;X)$ admits a Radon-Nikodym derivative in $L^\infty (\R_+ ;X)$.

Similarly, a Banach space $X$ has the {\em complete continuity property} if every operator $T:\L1\to X$ is locally Dunford-Pettis. Note that the Dunford-Pettis property for Banach spaces has also been defined in the literature, but is different from the complete continuity property, \cite[Definition 3.7.6]{MN91}.

Finally, a Banach space $X$ has the {\em Riemann-Lebesgue property} if every operator $T:\L1\to X$ is Riemann-Lebesgue. The Riemann-Lebesgue property for Banach spaces has been defined only recently, \cite{BuCh02}, and Definition \ref{tproperties} is perhaps the first instance where the Riemann-Lebesgue property is defined for a single operator. 
 
It has been recently shown that the complete continuity property and the Riemann-Lebesgue property for Banach spaces are equivalent, \cite{Iv07}. It is therefore natural to ask whether a similar result holds for single operators.

\begin{problem} \label{problemrldp}
Is every Riemann-Lebesgue operator $T:\L1 \to X$ a local Dunford-Pettis operator?
\end{problem}

The following theorem gives a characterization of Riemann-Lebesgue operators using only exponential functions.

\begin{theorem} \label{rl-char}
An operator $T:\L1 \to X$ is a Riemann-Lebesgue operator if and only if $\lim_{|\beta |\to\infty} \| T e_{\omega+i\beta} \| = 0$ for some/all $\omega >0$. 
\end{theorem}

\begin{proof}
Assume first that $T:\L1 \to X$ is a bounded operator satisfying \linebreak[4] $\lim_{|\beta |\to\infty} \| T e_{\omega+i\beta} \| = 0$ for some $\omega >0$. 

Let $0<a<\omega <b<\infty$, and define the closed strip $S:= \{ \lambda\in\C_+ : a \leq \Re \lambda \leq b \}$. The function 
\begin{eqnarray*}
f : S & \to & X ,\\
 \lambda & \mapsto & Te_\lambda ,
\end{eqnarray*}
is bounded, continuous on $S$, and analytic in the interior of $S$. By a standard argument from complex function theory (involving Vitali's theorem) and the assumption we obtain
$$
\lim_{|\beta |\to\infty } \| T(e_{\alpha +i\beta}) \| =0 ,
$$
for all $\alpha \in (a , b)$. Since $a\in (0,\omega )$ and $b\in (\omega ,\infty)$ are arbitrary, the above equation is true for every $\alpha \in (0,\infty )$.

Next, recall from \eqref{rlclosed} that the space of all $g\in\L1$ such that $\lim_{|\beta |\to\infty} \| T(e_{i\beta} g)\| = 0$ is closed in $\L1$. By the preceding argument, this space contains the set $\{ e_\alpha : \alpha >0 \}$. Since this set is total in $\L1$, by the Hahn-Banach theorem and by uniqueness of the Laplace transform, it therefore follows that $T$ is a Riemann-Lebesgue operator.

The other implication is trivial. 
\end{proof}

\begin{corollary} \label{rlcharls}
Let $F\in {\rm Lip}_0 (\R_+ ;X)$, and let $T_F:\L1 \to X$ be the corresponding bounded operator given by \eqref{TF}. Denote by $\widehat{dF}$ the {\em Laplace-Stieltjes transform} of $F$, that is, 
$$
\widehat{dF}(\lambda ) := \int_0^\infty e^{-\lambda t} \; dF(t) , \quad \lambda\in\C_+ .
$$
Then $T_F$ is a Riemann-Lebesgue operator if and only if
$$
\lim_{|\beta|\to\infty} \widehat{dF} (\omega +i\beta ) = 0 \text{ for some/all } \omega >0 .
$$ 
\end{corollary}

\begin{proof}
This is a direct consequence from Theorem \ref{rl-char} and the definition of the representing function.
\end{proof}

\section{Algebra homomorphisms $\L1 \to \cA$ and the norm-continuity problem}

In the following, we will equip $\L1$ with the usual convolution product given by
$$
(f*g)\, (t) = \int_0^t f(t-s) g(s) \; ds , \quad f, \, g\in\L1 .
$$
Then $\L1$ is a commutative Banach algebra with bounded approximate identity; for example, the net $(\lambda e_\lambda)_{\lambda\nearrow \infty}$ is an approximate identity bounded by $1$.

Let $\cA$ be a Banach algebra. If $(a(t))_{t>0}\subset \cA$ is a uniformly bounded and continuous semigroup, then the operator $T:\L1 \to \cA$ given by
\begin{equation} \label{Ta}
Tg = \int_0^\infty a(t) g(t) \; dt 
\end{equation}
is an algebra homomorphism as one easily verifies. Conversely, if $T:\L1 \to \cA$ is an algebra homomorphism, then $T$ is represented as above, but $(a(t))_{t>0}$ is a semigroup of {\em multipliers} on $\cA$ and the integral is to be understood in the sense of the strong topology of the multiplier algebra $\mathcal M(A)$; see, for example, \cite[Theorems 3.3 and 4.1]{Cj99}, \cite[Proposition 1.1]{Pe97}. We will state this result in a slightly different form, more convenient to us, using the notion of equivalent operators which we introduce here.\\

We call two operators $T:\L1\to X$ and $S:\L1\to Y$ {\em equivalent} if there exist constants $c_1$, $c_2\geq 0$ such that 
$$
\| Tg\|_X \leq c_1 \, \| Sg\|_Y \leq c_2 \, \| Tg\|_X \quad \text{for every } g\in\L1 .
$$
It is easy to check that properties like weak compactness, representability, the local Dunford-Pettis property and the Riemann-Lebesgue property are invariant under equivalence, that is, for example, if $T$ and $S$ are equivalent, then $T$ is representable if and only if $S$ is representable; one may prove that if $F_T$ and $F_S$ are the representing Lipschitz functions, then $F_T$ is differentiable almost everywhere if and only if $F_S$ is differentiable almost everywhere (use that difference quotients are images of multiples of characteristic functions). We point out that two operators $T$ and $S$ are equivalent if and only if ${\rm range}\, T^* = {\rm range}\, S^*$; compare with \cite[Theorem 1]{Em73}. \\

The proof of the following lemma is similar to the proof of \cite[Theorem 1]{Cj98}
(see also \cite[Corollary 4.3]{Ki98}, \cite[Theorem 10.1]{Ki00}, \cite[Theorem 1.1]{Bo01}).

\begin{lemma} \label{equivalenthomomorphism}
For every algebra homomorphism $T : \L1 \to \cA$ there exists a Banach space $X_0$, an equivalent algebra homomorphism $S:\L1 \to \cL (X_0)$ and a uniformly bounded $C_0$-semigroup $(S(t))_{t\geq 0}\subset \cL (X_0 )$ such that for every $g\in\L1$ 
$$
Sg = \int_0^\infty S(t) g(t) \; dt \quad \text{(integral in the strong sense)} .
$$
If $\cA \subset \cL (X)$ as a closed subspace, then $X_0$ can be chosen to be a closed subspace of $X$. 
\end{lemma}

\begin{proof}
We first assume that $\cA \subset\cL (X)$ as a closed subspace, and we put $R(\lambda ) := Te_\lambda \in \cL (X)$ ($\lambda\in\C_+$). Since $T$ is an algebra homomorphism, the function $R$ is a {\em pseudoresolvent}, that is, 
$$
R(\lambda ) - R(\mu ) = (\mu - \lambda ) R(\lambda ) R(\mu ) \quad \text{for every } \lambda, \, \mu\in\C_+ .
$$
This resolvent identity implies that 
$$
{\rm range}\, R (\lambda ) \text{  is independent of } \lambda\in\C_+ .
$$
We put
$$
X_0 := \overline{{\rm range}\, R(\lambda )}^{\|\cdot\|_X} \subset X .
$$
Clearly, $X_0$ is invariant under $R(\lambda )$, and since $\| \lambda e_\lambda \|_{L^1} = 1$ for every $\lambda >0$, we obtain the estimate
\begin{equation} \label{hy}
\| (\lambda R(\lambda ))^n \|_{\cL (X_0)} \leq \| (\lambda R(\lambda ))^n \|_{\cL (X)}
\leq \| T\| \quad \text{for every } \lambda >0 , \, n\geq 1 .
\end{equation}
By using this estimate (with $n=1$), for every $x\in X$ and every $\lambda\in\C_+$ one obtains
$$
\lim_{\mu\to\infty} \mu R(\mu ) R(\lambda) x = \lim_{\mu\to\infty} \frac{\mu R(\lambda ) x - \mu R(\mu) x }{\mu-\lambda} = R(\lambda )x ,
$$
which implies
$$
\lim_{\mu\to\infty} \mu R(\mu ) x = x \quad \text{for every } x\in X_0 .
$$
This relation and the resolvent identity imply that $R(\lambda )$ is injective on $X_0$ and the range of $R(\lambda )$ is dense in $X_0$. As a consequence, there exists a densely defined, closed operator $A$ on $X_0$ such that 
\begin{equation} \label{risresolvent}
R(\lambda )x = R(\lambda ,A) x \quad \text{for every } \lambda\in\C_+ ,\, x\in X_0 .
\end{equation}
By \eqref{hy} and the Hille-Yosida theorem, $A$ is the generator of a uniformly bounded $C_0$-semigroup $(S(t))_{t\geq 0}\subset \cL (X_0)$. Let $S:\L1 \to \cL (X_0 )$ be the operator defined by $Sg = \int_0^\infty S(t) g(t) \; dt$, where the integral is understood in the strong sense. Let $F\in {\rm Lip}_0 (\R_+ ; \cL (X))$ be the function representing $T$ (Riesz-Stieltjes representation). Then the equality \eqref{risresolvent} and the definition of $R$ imply
$$
\int_0^\infty e^{-\lambda t} \; dF (t)x = \int_0^\infty e^{-\lambda t} S(t)x \; dt \quad \text{for every } \lambda\in\C_+ , \, x\in X_0 .
$$
By the uniqueness of the Laplace-Stieltjes transform, we obtain
$$
\int_0^t S(s) x\; ds = F(t) x  \quad \text{for every } t\geq 0, \, x\in X_0 ,
$$
and hence
$$
(Sg)(x) = (Tg)(x) \quad \text{for every } g\in \L1, \, x\in X_0 .
$$
Clearly, this implies
$$
\| Sg \|_{\cL (X_0)} \leq \| Tg \|_{\cL (X)} \quad \text{for every } g\in \L1 .
$$
On the other hand, for every $x\in X$ one has
\begin{eqnarray*}
\| (Tg) (x) \| & = & \lim_{\mu\to\infty} \| \mu R(\mu ) (Tg) (x) \| \\
& = & \lim_{\mu\to\infty} \| (Sg) (\mu R(\mu ) x) \| \\
& \leq & \| Sg \|_{\cL (X_0 )} \, \sup_{\mu >0} \| \mu R(\mu )x\| \\
& \leq & \| Sg \|_{\cL (X_0 )} \, \| T\| \, \| x\| .
\end{eqnarray*}
The last two inequalities imply that $T$ and $S$ are equivalent. 

The general case can be reduced to the case $\cA\subset \cL (X)$ in the following way. First of all, we may assume without loss of generality that ${\rm range}\, T$ is dense in $\cA$. Since $\L1$ admits a bounded approximate identity, it is then easy to verify that also $\cA$ admits a bounded approximate identity. From this one deduces that the natural embedding 
$$
\cA \to \cL (\cA ) ,
$$
which to every element $a\in\cA$ associates the multiplier $M_a\in\cL (\cA )$ given by $M_a b = ab$, is an isomorphism onto its range. 
\end{proof}

\begin{remark}
It is in general an open problem to give conditions on an algebra homomorphism $T:\L1 \to \cA$ which imply that there exists an equivalent algebra homomorphism $S:\L1 \to \cL (X)$ on a Banach space $X$ having additional properties, for example, being reflexive, being an $L^p$ space etc. If $T$ has dense range, this is essentially the problem of representing $\cA$ as a closed subalgebra of $\cL (X)$. 
\end{remark}

The next lemma relates the norm-continuity problem with the problem of representability of homomorphisms $\L1 \to \cA$.

\begin{lemma} \label{rep-char}
An algebra homomorphism $T :\L1 \to \cA$ is representable if and only if there exists a uniformly bounded and continuous semigroup $(a(t))_{t>0}\subset\cA$ (no continuity condition at zero) such that $T$ is represented by \eqref{Ta}.
\end{lemma}

\begin{proof}  
The sufficiency part is trivial. 

So assume that $T$ is representable by some  $a \in L^{\infty}(\mathbb R_+, \cA)$. Without loss of generality we may assume that $T$ has dense range in $\cA$. By the proof of Lemma \ref{equivalenthomomorphism}, there exists an equivalent algebra homomorphism $S:\L1 \to \cL (\cA )$ which is (strongly) represented by a uniformly bounded $C_0$-semigroup $(S(t))_{t\geq 0}\subset\cL (\cA )$. Moreover,
$$
\int_0^\infty a(t)b \, g(t) \; dt = \int_0^\infty S(t)b \, g(t) \; dt \quad \text{for every } b\in\cA, \, g\in \L1 ,
$$
which in turn implies
$$
a(t) b = S(t) b \quad \text{for every } b\in\cA \text{ and almost every } t>0 .
$$
As a consequence, after changing $a$ on a set of measure zero, $(a(t))_{t>0}$ is a semigroup. 

Since $S$ is also representable, the semigroup $(S(t))_{t\geq 0}$ is measurable in $\cL (\cA )$, and hence immediately norm continuous by \cite[Theorem 9.3.1]{HiPh57}. Since $\cA$ admits a bounded approximate identity, one thus obtains that also $(a(t))_{t>0}$ is norm continuous.
\end{proof}

\begin{remark}
One can also prove that an algebra homomorphism $T:\L1 \to \cA$ is weakly compact if and only if it is represented by a uniformly bounded and continuous semigroup $(a(t))_{t\geq 0}\subset \cA$ (continuity at $0$ included!); compare with \cite{Ga92}, \cite{GaRaWh92}.
\end{remark}

Let $T:\L1 \to \cL (X)$ be an algebra homomorphism which is represented in the strong sense by a uniformly bounded $C_0$-semigroup $(T(t))_{t\geq 0}$ with generator $A$. By Lemma \ref{rep-char} above, $(T(t))_{t\geq 0}$ is immediately norm continuous if and only if $T$ is representable. By Theorem \ref{rl-char} and Corollary \ref{rlcharls}, the resolvent of $A$ satisfies the resolvent decay condition \eqref{res-cond} if and only if $T$ is Riemann-Lebesgue. Hence, by Lemma \ref{equivalenthomomorphism}, the norm-continuity problem can be reformulated in the following way.

\begin{problem}[Norm-continuity problem reformulated] \label{problemnc1} 
If $\cA$ is a Banach algebra and if $T:\L1 \to \cA$ is a Riemann-Lebesgue algebra homomorphism, is $T$ representable?
\end{problem}

We recall from the Introduction, that the norm-continuity problem has a negative answer in general, but that there are some positive answers in special cases. For example, by the representation theorem for $C^*$ algebras as subalgebras of $\cL (H)$ ($H$ a Hilbert space), and by Lemma \ref{equivalenthomomorphism}, the answer to Problem \ref{problemnc1} is positive if $\cA$ is a $C^*$ algebra. This follows from the result in \cite{Yo92}. 

The fact that the answer to Problem \ref{problemnc1} is in general negative follows from Matrai's example \cite{Ma08}. The aim of the following section is to construct suitable Riemann-Lebesgue homomorphisms and to deduce from this different counterexamples to Problem \ref{problemnc1} for which it is possible to control the resolvent decay along vertical lines. \\

At the same time, we are not able to answer the following variant of the norm-continuity problem. Observe that since Problem \ref{problemnc1} has in general a negative answer, this variant and Problem \ref{problemrldp} are not independent of each other.

\begin{problem}[Variant of the norm-continuity problem] 
If $\cA$ is a Banach algebra and if $T:\L1 \to \cA$ is a local Dunford-Pettis algebra homomorphism, is $T$ representable?
\end{problem}

We finish this section by collecting some basic properties of algebra homomorphisms $\L1\to \cA$ and their adjoints which are needed in the sequel.\\

For every $g\in \L1$, $h\in L^\infty (\R_+ )$ we define the {\em adjoint convolution} $g\circledast h \in L^\infty (\R_+ )$ by
$$
(g\circledast h)\, (t) = \int_0^\infty g(s) h(t+s) \; ds .
$$
With this definition, for every $f$, $g\in\L1$ and every $h\in L^\infty (\R_+ )$ we have the identities
$$
f\circledast (g\circledast h) = (f*g)\circledast h 
$$
and
$$
\langle f*g , h \rangle_{L^1 ,L^\infty} = \langle f , g\circledast h \rangle_{L^1 ,L^\infty} ,
$$
which will be frequently used in the following. The second identity explains the name of the product $\circledast$. From this identity one can also deduce that $\circledast$ is separately continuous on $(L^1 , {\rm weak})\times (L^\infty , {\rm weak}^*)$ with values in $(L^\infty , {\rm weak}^*)$.\\

Whenever $X$ is some Banach space, we denote by $B_X$ the closed unit ball in $X$.

\begin{lemma} \label{necessary-cond}
Let $T:\L1 \to \cA$ be an algebra homomorphism. Then the following are true:
\begin{itemize}
\item[(a)] The set $T^* B_{\cA^*} \subset L^\infty (\R_+ )$ is non-empty, convex, weak$^*$ compact and $T^* B_{\cA^*} = - T^* B_{\cA^*}$.
\item[(b)] If $\| T\| = 1$, then for every $g\in B_{L^1}$ and every $h\in T^* B_{\cA^*}\subset L^\infty (\R_+ )$ one has $g\circledast h\in T^* B_{\cA^*}$. In particular, ${\rm range}\, T^*$ is invariant under adjoint convolution.
\item[(c)] If $T$ is representable, then ${\rm range}\, T^* \subset C(0,\infty )$.
\item[(d)] If $T$ is represented (in the strong sense) by a bounded $C_0$-semigroup which is norm-continuous for $t>t_0$, then every function in  ${\rm range}\, T^*$ is continuous on $(t_0,\infty )$.
\end{itemize}
\end{lemma}

\begin{proof}
The properties in (a) are actually true for general bounded linear operators $T$ and do not depend on the spaces $\L1$ and $\cA$. The weak$^*$ compactness follows from Banach-Alaoglu and the other properties are true for any unit ball in a Banach space.

In order to prove (b), let $g\in B_{L^1}$ and $h = T^* a^* \in T^* B_{\cA^*}$ for some $a^*\in B_{\cA^*}$. Since $T$ is an algebra homomorphism, for every $f\in\L1$,
\begin{eqnarray*}
\langle f , g\circledast T^* a^* \rangle_{L^1,L^\infty} & = & \langle f*g , T^* a^* \rangle_{L^1 ,L^\infty} \\
& = & \langle Tf \, Tg , a^* \rangle_{\cA , \cA^*} \\
& =: & \langle Tf , Tg \, a^* \rangle_{\cA , \cA^*} ,
\end{eqnarray*}
so that $g\circledast T^* a^* = T^* (Tg \, a^*)$. However, $\| Tg \, a^*\|_{\cA^*} \leq \| Tg \|_{\cA} \, \| a^*\|_{\cA^*} \leq 1$, so that (b) is proved.

If $T$ is representable, then, by Lemma \ref{rep-char}, there exists a bounded norm-continuous semigroup $(a(t))_{t>0} \subset \cA$ such that 
$$
Tg = \int_0^\infty a(t) g(t) \; dt , \quad g\in\L1 .
$$
Hence, for every $a^*\in\cA^*$ and every $g\in \L1$,
\begin{eqnarray*}
\langle g , T^* a^* \rangle_{L^1 ,L^\infty } & = & \langle Tg , a^* \rangle_{\cA ,\cA^*} \\
& = & \int_0^\infty g(t) \langle a(t) , a^*\rangle_{\cA , \cA^*} \; dt .
\end{eqnarray*}
This implies $T^* a^* = \langle a(\cdot ) , a^*\rangle_{\cA , \cA^*} \in C(0,\infty )$ so that (c) is proved.

The last assertion is very similar to (c), if we use in addition that $\L1$ is the direct sum of $L^1 (0,t_0 )$ and $L^1 (t_0,\infty )$ (and $L^\infty (\R_+ )$ is the direct sum of the corresponding duals).
\end{proof}

\section{Construction of Riemann-Lebesgue homomorphisms}

This section is devoted to the main result of this article: we will describe a procedure how to construct Riemann-Lebesgue algebra homomorphisms $T:\L1 \to \cA$ for which one can estimate the norm decay of the pseudoresolvent $( Te_\lambda )_{\lambda\in\C_+}$ along vertical lines.\\

In the following, for every function $\varphi\in C(\R_+ )$ we define the domain
$$
\Sigma_\varphi := \{ \lambda\in\C : \Re \lambda > - \varphi (| \Im \lambda |) \} .
$$
The domains $\Sigma_\varphi$ are symmetric with respect to the real axis. Domains of the form $\Sigma_\varphi$ with $\lim_{\beta \to\infty} \varphi (\beta ) = \infty$ play an important role in connection with Riemann-Lebesgue algebra homomorphisms. The following proposition contains a necessary condition for algebra homomorphisms to be Riemann-Lebesgue. 

\begin{proposition} \label{necessary2}
If $\cA$ is a Banach algebra and if $T:\L1\to\cA$ is a Riemann-Lebesgue algebra homomorphism, then there exists a function $\varphi \in C(\R_+ )$ satisfying $\lim_{\beta \to \infty} \varphi (\beta ) = \infty$ such that for every $f\in {\rm range}\, T^*$ the
Laplace transform $\hat{f}$ extends to a bounded analytic function on $\Sigma_\varphi$. 
\end{proposition}

\begin{proof}
Assume that $T:\L1\to \cA$ is a Riemann-Lebesgue homomorphism. Then $\lim_{|\beta|\to\infty} \| Te_{2+i\beta}\| = 0$. Expanding the pseudoresolvent $\lambda \mapsto Te_\lambda$ in a power series at the points $2+i\beta$ with $\beta\in\R$, one easily verifies that this pseudoresolvent extends to a bounded analytic function in some domain $\Sigma_\varphi$, where $\varphi$ is as in the statement. 

If $f =T^* a^*\in {\rm range}\, T^*$, then for every $\lambda\in\C_+$ one has
\begin{eqnarray*}
\hat{f} (\lambda ) & = & \langle e_\lambda , f \rangle_{L^1,L^\infty} \\
& = & \langle Te_\lambda , a^* \rangle_{\cA , \cA^*} ,
\end{eqnarray*}
and therefore the Laplace transform $\hat{f}$ extends to a bounded analytic function on $\Sigma_\varphi$.
\end{proof}

The main result in this section goes in the opposite direction to Proposition \ref{necessary2}. Throughout the following, we put for every $\beta\in\R$
$$
\lambda_\beta := 2+i\beta ,
$$
and if $\varphi \in C(\R_+ )$ is a given nonnegative function, then we also put
$$
d_\beta := {\rm dist}\, (\lambda_\beta , \partial \Sigma_\varphi ) .
$$
It will not be necessary to make the dependence of $d_\beta$ on $\varphi$ explicit in the notation since the function $\varphi$ will always be clear from the context.

\begin{theorem} \label{main}
Let $f\in L^\infty (\R_+ )$ be a function such that its Laplace transform $\hat{f}$ extends to a bounded analytic function in some domain $\Sigma_\varphi$, where $\varphi \in C(\R_+ )$ satisfies $\inf \varphi >0$. 

Then there exists a Banach space $X$ which embeds continuously into $L^\infty (\R_+ )$ and which is left-shift invariant such that 
\begin{itemize}
\item[(i)] the left-shift semigroup $(T(t))_{t\geq 0}$ on $X$ is bounded and strongly continuous,
\item[(ii)] the resolvent of the generator $A$ satisfies the decay estimate
\begin{equation} \label{decayt}
\| R(\lambda_\beta , A)\|_{\cL (X)} \leq C\, \frac{\log d_\beta}{d_\beta} \quad \text{for every } \beta\in\R ,
\end{equation}
\item[(iii)] if $T :\L1\to \cL (X)$ is the algebra homomorphism which is represented (in a strong sense) by $(T(t))_{t\geq 0}$, then $f\in {\rm range}\, T^*$, and
\item[(iv)] the following inclusion holds:
$$
X\subset \overline{\L1 \circledast f}^{(L^\infty , {\rm weak}^*)}, 
$$
\end{itemize}
If, in addition, the function 
\begin{eqnarray*}
\C_+ &\to& L^\infty (\R_+ ), \\
\lambda &\mapsto& e_\lambda\circledast f ,
\end{eqnarray*}
extends analytically to $\Sigma_\varphi$ and if there exists some $r\in (0,1)$ such that
\begin{equation} \label{decayf}
\sup_{\lambda\in B(\lambda_\beta , r \,d_\beta)} \| e_\lambda\circledast f\|_\infty \leq C \, \frac{1}{d_\beta} \quad \text{for every } \beta\in\R .
\end{equation}
then the space $X$ can be chosen in such a way that the resolvent satisfies the stronger estimate 
$$
\| R(\lambda_\beta , A)\|_{\cL (X)} \leq C\, \frac{1}{d_\beta} \quad \text{for every } \beta\in\R .
$$
\end{theorem}

\begin{remark}
The condition $\inf \varphi >0$ in the above theorem simplifies the proof in some places but is not essential. Moreover, it can always be achieved by rescaling the function $f$ or the semigroup $(T(t))_{t\geq 0}$.
\end{remark}

The important points in the above theorem are the statements that the resolvent decay condition \eqref{res-cond} is satisfied as soon as $\lim_{\beta\to\infty} \varphi (\beta ) = \infty$, and that at the same time $f\in {\rm range}\, T^*$.

Thus, if we are able to find a function $f\in L^\infty (\R_+ )$ such that its Laplace transform $\hat{f}$ extends to a bounded analytic function on $\Sigma_\varphi$, where $\varphi\in C(\R_+ )$ satisfies $\lim_{\beta\to\infty} \varphi (\beta ) = \infty$, and such that $f$ is {\em not} continuous on $(0,\infty )$, then the Riemann-Lebesgue operator from Theorem \ref{main} is not representable by Lemma \ref{necessary-cond} (c), that is, the semigroup $(T(t))_{t\geq 0}$ is not immediately norm continuous (Lemma \ref{rep-char}). In other words, the existence of such a function $f$ solves the norm-continuity problem. It is straightforward to check that the characteristic function $f=1_{[0,1]}$ provides such an example. This and another example will be discussed in Section \ref{sectionexamples}.

For these examples it will be of substantial interest that Theorem \ref{main} also gives an estimate of the resolvent $R(\cdot ,A)$ along vertical lines, in terms of the Laplace transform $\hat{f}$, the decay of the function $\lambda\mapsto e_\lambda\circledast f$ and the growth of the function $\varphi$. We point out that a decay condition weaker than \eqref{decayf} is always true, as we will prove in Lemma \ref{conditiongamma} below. We do not know whether the decay condition \eqref{decayf} is always satisfied.\\

The rest of this section will be devoted to the proof of Theorem \ref{main}, that is, to the construction of the Banach space $X$ and the algebra homomorphism $T :\L1 \to \cL (X)$. The space $X$ will be a closed subspace of an appropriate Banach space $M$ which is continuously embedded into $L^\infty (\R_+ )$; we will first construct $M$ by constructing its unit ball.

\begin{lemma} \label{l1}
Let $(f_n)\subset L^\infty (\R_+ )$ be a bounded sequence and define the set
\begin{equation} \label{bm}
B_M := \overline{\big\{ \sum_n g_n \circledast f_n : (g_n)\in B_{l^1 (\L1 )} \big\} }^{(L^\infty , {\rm weak}^*)} \subset L^\infty (\R_+ ).
\end{equation}
Then:
\begin{itemize}
\item[(a)] The set $B_M$ is non-empty, convex, weak$^*$ compact and $B_M = - B_M$.
\item[(b)] For every $g\in B_{L^1}$ and every $h\in B_M$ one has $g\circledast h\in B_M$.
\item[(c)] For every $n$ one has $f_n\in B_M$. 
\end{itemize}
\end{lemma}

\begin{proof}
The properties in (a) are either trivial or easy to check. 

Next, let $g\in B_{L^1}$ and $h\in B_M$. Assume first that $h= \sum_n g_n\circledast f_n$ for some sequence $(g_n) \in B_{l^1 (\L1 )}$. Then 
$$
g\circledast h = \sum_n g\circledast (g_n \circledast f_n ) = \sum_n (g *g_n) \circledast f_n .
$$
Since $(g*g_n)\in B_{l^1 (\L1 )}$, this implies $g\circledast h\in B_M$. For general $h\in B_M$, by the definition of $B_M$, there exists a net $(h_\alpha)\subset B_M$, $h_\alpha = \sum_n g_n^\alpha \circledast f_n$ for some $(g_n^\alpha)\in B_{l^1 (\L1 )}$, such that $w^*-\lim_\alpha h_\alpha = h$. However, then $g\circledast h = w^*-\lim_\alpha g\circledast h_\alpha$ as one easily verifies. Since $B_M$ is weak$^*$ closed, we have proved (b).

By definition of $B_M$, one has $g\circledast f_n\in B_M$ for every $g\in B_{\L1}$. Taking an approximate unit $(g_j)$ in $B_{\L1}$, one easily shows $w^*-\lim_j g_j \circledast f_n = f_n$. Since $B_M$ is weak$^*$ closed, this proves (c).
\end{proof}

For a bounded sequence $(f_n)\subset L^\infty (\R_+ )$ we define the set $B_M\subset L^\infty (\R_+ )$ as in \eqref{bm}, and then we put
\begin{equation} \label{spacem}
M:= \R_+ \, B_M .
\end{equation}
Then $M$ is a (in general nonclosed) subspace of $L^\infty (\R_+ )$ and becomes a normed space when it is equipped with the Minkowski norm
$$
\| h\|_M := \inf \{ \lambda > 0 : h \in \lambda \, B_M \} .
$$
When $M$ is equipped with this Minkowski norm, then $B_M$ is the unit ball of $M$, and there is no ambiguity with our previously introduced notation. Moreover, $M$ embeds continuously into $L^\infty (\R_+ )$.

By a result by Dixmier, $M$ is a dual space, and in particular $M$ is a Banach space, \cite{Di48}. To be more precise, consider the natural embedding $S:\L1 \to M^*$ given by
\begin{equation} \label{s}
\langle Sg , m \rangle_{M^* ,M} := \langle g , m \rangle_{L^1 , L^\infty} ,
\end{equation}
and let 
\begin{equation} \label{ms}
M_* := \overline{{\rm range}\, S}^{M^*} .
\end{equation}
Then we have the following result; the short proof follows \cite[Proof of Theorem 1]{Kj77}.

\begin{lemma} \label{kaijser}
The space $M$ is isometrically isomorphic to $\mpd$, that is, to the dual of $M_*$.
\end{lemma}

\begin{proof}
The key point is the fact that, by construction, $B_M$ is weak$^*$ compact in $L^\infty (\R_+ )$. By the definition of the operator $S$ and by the definition of the space $M_*$, this implies that the unit ball $B_M$ is compact with respect to the $\sigma (M,M_* )$ topology. 

Consider the contraction $J: M \to \mpd$ which maps every $m\in M$ to the functional $Jm\in\mpd$ given by $\langle Jm , m^* \rangle_{\mpd, M_*} := \langle m, m^* \rangle_{M, M^*}$. The space $M_*$ separates the points in $M$ because the space $\L1$ separates the points in $M \subset L^\infty (\R_+ )$. Therefore, the operator $J$ is injective.

Next, let $m^{**}\in\mpd$ and assume for simplicity that $\| m^{**} \|_{\mpd} =1$. By Hahn-Banach, we may consider $m^{**}$ also as an element in $B_{M^{**}}$. By Goldstine's theorem, there exists a net $(m_\alpha )\subset B_M$ which converges to $m^{**}$ in $\sigma ( M^{**} , M^* )$. Since $B_M$ is compact with respect to the $\sigma (M,M_* )$ topology, there exists $m\in B_M$ such that 
$$
\langle m ,m^* \rangle_{M,M^*} = \lim_\alpha \, \langle m_\alpha , m^* \rangle_{M,M^*} = \langle m^{**} , m^* \rangle_{\mpd, M_*} \text{ for every } m^*\in M_* .
$$
Hence, $Jm = m^{**}$ and $\| m\|_M \leq 1 = \| Jm \|_{\mpd}$. We have thus proved that $J$ is also surjective and isometric.
\end{proof}

\begin{remark}
Using only the definitions of the operators $J$ and $S$, it is straightforward to verify that $S^*J$ is the natural embedding of $M$ into $L^\infty (\R_+ )$ and that
\begin{equation} \label{ranges}
{\rm range}\, S^* = M .
\end{equation}
\end{remark}

\begin{lemma} \label{l2}
The space $M$ is an $\L1$ module in a natural way: for every $g\in\L1$ and every $m\in M$ ($\subset L^\infty (\R_+ )$) the adjoint convolution $g\circledast m$ belongs to $M$ and 
$$
\| g\circledast m\|_M \leq \| g\|_{L^1} \, \| m\|_M .
$$
\end{lemma}

\begin{proof}
By Lemma \ref{l1} (b), for every nonzero $g\in\L1$ and $m\in M$ one has $\frac{g}{\| g\|_{L^1}} \circledast \frac{m}{\| m\|_M} \in B_M$. The claim follows immediately. 
\end{proof}

In the following, we will always consider $M$ as an $\L1$ module via the adjoint convolution. Note that together with $M$ also the dual space $M^*$ is an $\L1$ module if for every $g\in\L1$ and every $m^*\in M^*$ we define the product $g* m^*\in M^*$ by
$$
\langle g * m^* , m\rangle_{M^* ,M} := \langle m^* , g\circledast m \rangle_{M^* , M} , \quad m\in M .   
$$
We use again the notation $*$ for the adjoint of the adjoint convolution. If $M=L^\infty (\R_+ )$, then the product $*$ coincides with the usual convolution in $\L1 \subset L^\infty (\R_+ )^*$ and there is no ambiguity in the notation.

\begin{lemma} \label{l3}
The natural embedding $S:\L1 \to M^*$ defined in \eqref{s} is an $\L1$ module homomorphism. The space $M_*$ is an $\L1$ submodule of $M^*$.
\end{lemma}

\begin{proof}
For every $g$, $h\in\L1$ and every $m\in M$ one has
\begin{eqnarray*}
\langle S(g*h) , m\rangle_{M^* ,M} & = & \langle g*h , m \rangle_{L^1 , L^\infty} \\
& = & \langle h , g\circledast m \rangle_{L^1 , L^\infty} \\
& = & \langle Sh , g\circledast m \rangle_{M^* , M} \\
& = & \langle g*Sh , m \rangle_{M^* , M} .
\end{eqnarray*}
Since this equality holds for every $m\in M$, this proves $S(g*h) = g * Sh$ for every $g$, $h\in\L1$, and therefore $S$ is an $\L1$ module homomorphism. At the same time, this equality proves that the closure of the range is a submodule of $M^*$.
\end{proof}

We omit the proof of the following lemma which is straightforward.

\begin{lemma} \label{l3a}
Let $M_*$ be defined as in \eqref{ms}. The natural embedding
$$
T_* : \L1 \to \cL (M_*)
$$
given by $T_* g (m^*) := g* m^*$, $m^*\in M_*$, is an algebra homomorphism. 
\end{lemma}

The following lemma allows us to calculate $\| T_* g\|_{\cL (M_*)}$ in terms of the sequence $(f_n)$.
 
\begin{lemma} \label{l4}
Let $(f_n)\subset L^\infty (\R_+ )$ be a bounded sequence, let $M_*$ be defined as in \eqref{ms}, and let $T_* : \L1\to \cL (M_*)$ be the induced algebra homomorphism from Lemma \ref{l3a}. Then, for every $g\in\L1$, one has
\begin{equation}
\| T_* g\|_{\cL (M_*)} = \sup_n \| g\circledast f_n\|_M .
\end{equation}
\end{lemma}

\begin{proof}
Let $g\in\L1$. Then, by the definition of $T_*$, $S$, and by the definition of $M_*$,
\begin{eqnarray*}
\| T_* g\|_{\cL (M_*)} & = & \sup_{m^*\in B_{M_*}} \| g * m^* \|_{M^*} \\
& = & \sup_{m^*\in B_{M_*}} \, \sup_{m\in B_M} | \langle g * m^* , m\rangle_{M^* ,M} | \\
& = & \sup_{\substack{h\in \L1 \\ \| Sh\|_{M^*} \leq 1}} \, \sup_{m\in B_M} | \langle g * Sh , m\rangle_{M^* ,M} | .
\end{eqnarray*}
Since $S$ is an $\L1$ module homomorphism, and by the definition of $S$,
\begin{eqnarray*}
\| T_* g\|_{\cL (M_*)} & = & \sup_{\substack{h\in \L1 \\ \| Sh\|_{M^*} \leq 1}} \, \sup_{m\in B_M} | \langle S(g * h) , m\rangle_{M^* ,M} | \\
& = & \sup_{\substack{h\in \L1 \\ \| Sh\|_{M^*} \leq 1}} \, \sup_{m\in B_M} | \langle g * h , m\rangle_{L^1 ,L^\infty} | .
\end{eqnarray*}
Since, by definition, $\{ \sum_n g_n \circledast f_n : (g_n)\in B_{l^1 (\L1 )} \}$ is weak$^*$ dense in $B_M$ (with respect to the weak$^*$ topology in $L^\infty (\R_+ )$), and since $M_*$ is norming for $M$ by Lemma \ref{kaijser}, we can continue to compute
\begin{eqnarray*}
\| T_* g\|_{\cL (M_*)} & = & \sup_{\substack{h\in \L1 \\ \| Sh\|_{M^*} \leq 1}} \, \sup_{(g_n)\in B_{l^1 (\L1 )}} | \langle g * h , \sum_n g_n \circledast f_n \rangle_{L^1 ,L^\infty} | \\
& = & \sup_{\substack{h\in \L1 \\ \| Sh\|_{M^*} \leq 1}} \, \sup_{(g_n)\in B_{l^1 (\L1 )}} | \langle h , \sum_n g_n \circledast (g\circledast f_n) \rangle_{L^1 ,L^\infty} | \\
& = & \sup_{\substack{h\in \L1 \\ \| Sh\|_{M^*} \leq 1}} \, \sup_{(g_n)\in B_{l^1 (\L1 )}} | \langle Sh , \sum_n g_n \circledast (g\circledast f_n) \rangle_{M^* ,M} | \\
& = & \sup_{(g_n)\in B_{l^1 (\L1 )}} \| \sum_n g_n \circledast (g\circledast f_n) \|_M .
\end{eqnarray*}
This immediately implies 
\begin{eqnarray*}
\| T_* g\|_{\cL (M_*)} & \geq & \sup_{h\in B_{\L1}} \| h \circledast (g \circledast f_n ) \|_M  \\
& = & \sup_{h\in B_{\L1}} \| ( h * g ) \circledast f_n  \|_M \quad \text{for every } n . 
\end{eqnarray*}
By putting $h= \lambda e_\lambda$, letting $\lambda \to\infty$, and using Lemma \ref{l2}, we obtain
$$
\| T_* g\|_{\cL (M_*)} \geq \| g \circledast f_n  \|_M \quad \text{for every } n .
$$
On the other hand, by Lemma \ref{l2} again,
\begin{eqnarray*}
\| T_* g\|_{\cL (M_*)} & \leq & \sup_{(g_n)\in B_{l^1 (\L1 )}} \sum_n \| g_n\|_{L^1} \, \| g\circledast f_n \|_M \\
& \leq & \sup_{n} \| g\circledast f_n \|_M .
\end{eqnarray*}
The preceding two estimates imply the claim.
\end{proof}

The operator $T_*$ from Lemma \ref{l3a} will be equivalent to the operator we are looking for in Theorem \ref{main}. However, so far we have not said anything about the sequence $(f_n)\subset L^\infty (\R_+ )$ which served for the construction of $M$, and which allows us by Lemma \ref{l4} to obtain the desired resolvent estimate in Theorem \ref{main}. \\ 

It remains to explain how the sequence $(f_n)$ is constructed in order to prove Theorem \ref{main}. For the time being, let $f\in L^\infty (\R_+ )$ be a fixed function, and suppose that the Laplace transform $\hat{f}$ extends analytically to a bounded function on $\Sigma_\varphi$, where $\varphi\in C(\R_+ )$ satisfies $\inf \varphi >0$.\\

In order to simplify the notation, we define for every $k\in\Z$
\begin{eqnarray*}
& & \lambda_k := 2+ik , \\
& & d_k := {\rm dist}\, (\lambda_k , \partial\Sigma_\varphi ) , \text{ and} \\
& & \tilde{e}_k = e_{\lambda_k} ,
\end{eqnarray*}
and we will choose numbers
$$
c_k >0 
$$
depending on the functions $\lambda\mapsto e_\lambda\circledast f$ and $\varphi$; see Proposition \ref{l5} below for the precise definition of $c_k$.

We define inductively for $n\geq 1$ and ${\bf k} \in\Z^{n+1}$
$$
\tilde{e}_{\bf k} := \tilde{e}_{\bar{\bf k}} * \tilde{e}_{k_{n+1}} = \tilde{e}_{k_1} *\, \dots \,* \tilde{e}_{k_{n+1}}   
$$
and
$$
c_{\bf k} := c_{\bar{\bf k}} \cdot c_{k_{n+1}} = c_{k_1} \cdot \, \dots \, \cdot c_{k_{n+1}} ,
$$
where $\bar{\bf k}\in\Z^n$ is such that ${\bf k}=(\bar{\bf k},k_{n+1})$.

Then, for every $n\geq 1$ and every ${\bf k}\in\Z^n$ we put
\begin{equation} \label{fkdef}
f_{\bf k} := \frac{\tilde{e}_{\bf k} \circledast f}{c_{\bf k}} = \frac{\tilde{e}_{k_1} * \, \dots \, * \tilde{e}_{k_n}}{c_{k_1} \cdot \, \dots \, \cdot c_{k_n}} \circledast f .
\end{equation}
Finally, we set
$$
f_\infty := f ,
$$
and
$$
I:= \{\infty \} \cup \bigcup_{n\geq 1} \Z^n ,
$$
and we will define the unit ball $B_M$, the space $M$ and the space $X$ starting from the family
$$
( f_{\bf k} )_{{\bf k}\in I } .
$$

\begin{proposition} \label{l5}
Let $f\in L^\infty (\R_+ )$ be such that the Laplace transform $\hat{f}$ extends to a bounded analytic function in $\Sigma_\varphi$, where $\varphi\in C(\R_+ )$ satisfies $\inf \varphi >0$. Let $r\in (0,\frac14 )$ be arbitrary. For every $k\in\Z$ we put
\begin{equation} \label{defck}
c_k = \frac{4}{r}\, \frac{\log d_k}{d_k} .
\end{equation}
Then the family $(f_{\bf k})_{{\bf k}\in I}$ given by \eqref{fkdef} is bounded in $L^\infty (\R_+ )$.

The same is true if the condition \eqref{decayf} is satisfied and if we then put, for every $k\in\Z$,
$$
c_k = \frac{4}{r} \, \frac{1}{d_k} .
$$
\end{proposition}

The proof of Proposition \ref{l5} is based on the following series of four lemmas. The statement and the proof of the following lemma should be compared to \cite[Lemmas 4.6.6, 4.7.9]{ABHN01}.

\begin{lemma} \label{laplacetoe}
Let $f\in L^\infty (\R_+ )$ be such that the Laplace transform $\hat{f}$ extends to a bounded analytic function in $\Sigma_\varphi$, where $\varphi\in C(\R_+ )$ satisfies $\inf \varphi >0$. Then also the function $\lambda\mapsto e_\lambda\circledast f$, $\C_+ \mapsto BUC (\R_+ )$ extends to a bounded analytic function in $\Sigma_\varphi$.
\end{lemma}

\begin{proof}
For every $t\in\R_+$ and every $\lambda\in\C_+$ one has
\begin{eqnarray*}
(e_\lambda \circledast  f) \, (t) & = & \int_0^\infty e^{-\lambda s} f(t+s) \; ds \\
& = & e^{\lambda t} \hat{f} (\lambda ) - \int_0^t e^{\lambda (t-s)} f(s) \; ds .
\end{eqnarray*}
From this identity we obtain first that for every fixed $t\in\R_+$ the function $\lambda\mapsto (e_\lambda \circledast f)\, (t)$ extends to an analytic function on $\Sigma_\varphi$, and we obtain second for every $t\in\R_+$ and every $\lambda\in\Sigma_\varphi$ the estimate
\begin{equation} \label{eft}
|(e_\lambda \circledast f ) \, (t) | \leq \left\{ \begin{array}{ll}
\frac{\| f\|_\infty}{|{\rm Re}\, \lambda|} & \text{if } {\rm Re}\,\lambda >0 , \\[2mm]
\frac{\| f\|_\infty}{|{\rm Re}\, \lambda|} + \| \hat{f}\|_\infty & \text{if } {\rm Re}\,\lambda <0 .
\end{array} \right.
\end{equation}
By assumption, there exists $0<\alpha\leq 1$ such that $\inf\varphi >\alpha$. The above estimate immediately yields
$$
\sup_{\lambda\in\Sigma_\varphi \atop |{\rm Re}\, \lambda|\geq \frac{\alpha}{2}} \sup_{t\in\R_+} |(e_\lambda \circledast f)\, (t)| \leq \frac{2\, \| f\|_\infty}{\alpha} +\| \hat{f}\|_\infty .
$$
In order to show that the function $(e_\lambda \circledast f)\, (t)$ is bounded in the strip $\{ \lambda\in\C : |{\rm Re}\,\lambda |\leq \frac{\alpha}{2} \}$ (with a bound independent of $t\in\R_+$) we can argue as follows. For every $\beta\in\R$, by the maximum principle and by the estimate \eqref{eft},
\begin{eqnarray*}
& & \sup_{| \lambda - i\beta |\leq \alpha} \big| (e_\lambda \circledast f)\, (t) \, \big( 1+\frac{(\lambda -i\beta)^2}{\alpha^2} \big) \big| \\
& = & \sup_{|\lambda - i\beta |=\alpha} \big| (e_\lambda \circledast f)\, (t) \, \big( 1+\frac{(\lambda -i\beta)^2}{\alpha^2} \big) \big| \\
& \leq & \frac{4\, \| f\|_\infty}{\alpha} + 4\, \| \hat{f}\|_\infty .
\end{eqnarray*}
Hence, for every $t\in\R_+$ and every $\beta\in\R$
$$
\sup_{| \lambda - i\beta |\leq \frac{\alpha}{2}} | (e_\lambda \circledast f)\, (t) | \leq \frac{6\, \| f\|_\infty}{\alpha} + 6\, \| \hat{f}\|_\infty ,
$$
which yields the desired estimate in the strip $\{ \lambda\in\C : |{\rm Re}\,\lambda |<\frac{\alpha}{2} \}$. So we finally obtain
$$
\sup_{\lambda\in\Sigma_\varphi} \sup_{t\in\R_+} |(e_\lambda \circledast f)\, (t)| <\infty ,
$$
and in particular the function $\lambda\mapsto e_\lambda\circledast f$ is bounded on $\Sigma_\varphi$ with values in $BC (\R_+ )$. Now one may argue as in the proof of \cite[Corollary A.4]{ABHN01}. Pointwise analyticity and uniform boundedness imply, by \cite[Proposition A.3]{ABHN01}, that the function $\lambda\mapsto e_\lambda\circledast f$ is bounded and analytic on $\Sigma_\varphi$ with values in $BC (\R_+ )$. Since $e_\lambda\circledast f\in BUC (\R_+ )$ for every $\lambda\in\C_+$, by the identity theorem for analytic functions (see, for example, the version in \cite[Proposition A.2]{ABHN01}), we finally obtain the claim. 
\end{proof}

The main argument in the proof of the following lemma (the two constants theorem) is also used in \cite[Proof of Theorem 5.3]{Il07}, but the following lemma gives a better estimate. Recall that $\lambda_\beta = 2+i\beta$ and $d_\beta = {\rm dist}\, (\lambda_\beta , \partial\Sigma_\varphi )$.

\begin{lemma} \label{conditionalpha}
Let $X$ be some Banach space and let $\varphi\in C(\R_+ )$ be a nonnegative function. Let $h:\Sigma_\varphi \to X$ be a bounded analytic function satisfying the estimate
$$
\| h(\lambda ) \| \leq \frac{C}{{\rm Re}\, \lambda} \quad \text{for every } \lambda\in\C_+ \text{ and some } C\geq 0 . 
$$
Then for every $r \in (0,\frac14 )$ there exists $C_r \geq 0$ such that for every $\beta\in\R$ 
$$
\sup_{\lambda\in B(\lambda_\beta ,r \frac{d_\beta}{\log d_\beta} )} \| h(\lambda ) \| \leq C_r \, \frac{\log d_\beta}{d_\beta} .
$$
\end{lemma}

\begin{proof}
We may assume that the constant $C$ from the hypothesis satisfies $C\geq \| h\|_\infty$. 
Fix $r\in (0,\frac14 )$. We may in the following consider only those $\beta\in\R$ for which $4< \log d_\beta$. For the other $\beta$, the estimate in the claim becomes trivial if the constant $C_r$ is chosen sufficiently large.

Let 
$$
\Omega := \{ \lambda\in\C : |{\rm Re}\,\lambda | , |{\rm Im}\,\lambda | < 1 \} ,
$$
and let $\Gamma_0 := \{ \lambda \in\partial\Omega : {\rm Re}\, \lambda =1 \}$ and $\Gamma_1 := \partial\Omega \setminus \Gamma_0$. 

By the two constants theorem, for every analytic function $g:\Omega \to X$ having a continuous extension to $\bar{\Omega}$ and satisfying the boundary estimate
$$
\| g(\lambda )\| \leq C_i  \quad \text{if } \lambda\in\Gamma_i \,\, (i=0,\, 1) ,
$$
one has the estimate
$$
\| g(\lambda ) \| \leq C_0^{w(\lambda )} \, C_1^{1-w(\lambda )} \quad \text{for every } \lambda\in\Omega ,
$$
where $w = w_\Omega (\cdot , \Gamma_0 )$ is the harmonic measure of $\Gamma_0$ with respect to $\Omega$, that is, $w: \Omega \to [0,1]$ is the harmonic function satisfying $w=1$ on $\Gamma_0$ and $w=0$ on $\Gamma_1$.

For every $\beta\in\R$ with $\log d_\beta >4$ we apply this two constants theorem to the function given by
$$
g(\lambda ) = h(\lambda_\beta + \frac{d_\beta}{4} (\lambda - 1+ \frac{1}{\log d_\beta}) ) , \quad \lambda\in\bar{\Omega} , 
$$
which satisfies by assumption the estimates
$$
\| g(\lambda ) \| \leq \left\{ \begin{array}{ll}
C \, \frac{\log d_\beta}{d_\beta} & \text{if } \lambda\in\Gamma_0 , \text{ and} \\[2mm]
C & \text{if } \lambda\in\Gamma_1 .
                               \end{array} \right.
$$
We then obtain
$$
\| g (\lambda ) \| \leq C\, \big( \frac{\log d_\beta}{d_\beta} \big)^{w(\lambda )} \quad \text{for every } \lambda\in\Omega .
$$

By the Schwarz reflection principle, the function $1-w$ extends to a harmonic function in the rectangle $\{ \lambda \in\C : -1 < {\rm Re}\,\lambda < 3$, $|{\rm Im}\,\lambda| <1 \}$, and in particular the function $w$ is continuously differentiable there. We can therefore find a constant $c>0$ such that 
$$
\inf_{\lambda\in B(1-\frac{1}{\log d_\beta} , 4r \, \frac{1}{\log d_\beta} )} w(\lambda ) \geq 1- \frac{c\, (1+4r)}{\log d_\beta} .
$$ 

Combining the preceding two estimates, we obtain
\begin{eqnarray*}
\sup_{\lambda\in B(\lambda_\beta , r \, \frac{d_\beta}{\log d_\beta} )} \| h(\lambda ) \| & = & \sup_{\lambda\in B(1-\frac{1}{\log d_\beta} , 4r \, \frac{1}{\log d_\beta} )}  \| g(\lambda ) \| \\
& \leq & C \, \sup_{\lambda\in B(1-\frac{1}{\log d_\beta} , 4r \, \frac{1}{\log d_\beta} )} \big( \frac{\log d_\beta}{d_\beta} \big)^{w(\lambda )} \\
& \leq & C \, \big( \frac{\log d_\beta}{d_\beta} \big)^{1-\frac{c\, (1+4r)}{\log d_\beta}} \\
& = & C \, e^{c\, (1+4r) (1- \frac{\log \log d_\beta}{\log d_\beta})} \, \frac{\log d_\beta}{d_\beta} \\
& \leq & C_r \, \frac{\log d_\beta}{d_\beta} .
\end{eqnarray*}
The claim is proved.
\end{proof}

\begin{lemma} \label{conditiongamma}
Let $f\in L^\infty (\R_+ )$ be such that the Laplace transform $\hat{f}$ extends to a bounded analytic function in $\Sigma_\varphi$, where $\varphi\in C(\R )^+$ satisfies $\inf\varphi >0$. Then for every $r\in (0,\frac14 )$ there exists $C_r \geq 0$ such that 
\begin{eqnarray*}
\sup_{\lambda\in B(\lambda_k , r \frac{d_k}{\log d_k})} \| e_\lambda \circledast f\|_\infty \leq C_r \, \frac{\log d_k}{d_k} \quad \text{for every } k\in\Z .
\end{eqnarray*}
\end{lemma}

\begin{proof}
Since $\| e_\lambda \circledast f\|_\infty \leq \frac{C}{{\rm Re}\, \lambda}$ for every $\lambda\in\C_+$, this lemma is a direct consequence of Lemma \ref{laplacetoe} and Lemma \ref{conditionalpha}.
\end{proof}

The following is a consequence of the resolvent identity and should probably be known. We will give the easy proof here.

\begin{lemma} \label{l6}
For every $n\geq 1$, every $\lambda_1$, $\dots$, $\lambda_n \in\C_+$, and every closed path $\Gamma\subset\C_+$ such that $\lambda_1$, $\dots$, $\lambda_n$ are in the interior of $\Gamma$ one has 
\begin{equation} \label{multi-res}
e_{\lambda_1}* \, \dots \, *e_{\lambda_n} = \frac{(-1)^{n+1}}{2\pi i} \int_\Gamma \frac{e_\lambda}{(\lambda -\lambda_1) \cdot \, \dots \, \cdot (\lambda -\lambda_n)} \; d\lambda .
\end{equation}
\end{lemma}

\begin{proof}
The proof goes by induction on $n$. 

If $n=1$, then the formula \eqref{multi-res} is just Cauchy's integral formula.

So assume that the formula \eqref{multi-res} is true for some $n\geq 1$. Let $\lambda_1$, $\dots$, $\lambda_n$, $\lambda_{n+1}\in\C_+$, and let $\Gamma\subset\C_+$ be a closed path such that $\lambda_1$, $\dots$, $\lambda_n$, $\lambda_{n+1}$ are in the interior of $\Gamma$. Then, by the resolvent identity and the induction hypothesis,
\begin{eqnarray*}
e_{\lambda_1}* \, \dots \, *e_{\lambda_n} * e_{\lambda_{n+1}} & = & \frac{(-1)^{n+1}}{2\pi i} \int_\Gamma \frac{e_\lambda * e_{\lambda_{n+1}}}{(\lambda -\lambda_1) \cdot \, \dots \, \cdot (\lambda -\lambda_n)} \; d\lambda \\
& = & \frac{(-1)^{n+2}}{2\pi i} \int_\Gamma \frac{e_\lambda}{(\lambda -\lambda_1) \cdot \, \dots \, \cdot (\lambda -\lambda_n) \, (\lambda - \lambda_{n+1}) } \; d\lambda + \\
& & + e_{\lambda_{n+1}}\, \frac{(-1)^{n+1}}{2\pi i} \int_\Gamma \frac{1}{(\lambda -\lambda_1) \cdot \, \dots \, \cdot (\lambda -\lambda_n) \, (\lambda - \lambda_{n+1}) } \; d\lambda .
\end{eqnarray*}
For the induction step it suffices to show that the second integral on the right-hand side of this equality vanishes. In order to see that this integral vanishes, we replace the path $\Gamma$ by a circle centered in $0$ and having radius $R>0$ large enough, without changing the value of the integral. A simple estimate then shows that
$$
\big| \int_{|\lambda |=R} \frac{1}{(\lambda -\lambda_1) \cdot \, \dots \, \cdot (\lambda -\lambda_n) \, (\lambda - \lambda_{n+1}) } \; d\lambda \big| = O(R^{-n}) \text{ as } R\to\infty .
$$
Since the left-hand side is independent of $R$ and since $n\geq 1$, by letting $R\to\infty$, we obtain that the integral above is zero. 
\end{proof}

\begin{proof}[Proof of Proposition \ref{l5}]
It will be convenient in this proof to define the function $h(s) := \frac{s}{\log s}$, $s\geq 2$. Then 
$$
c_k = \frac{4}{r} \frac{1}{h(d_k)} \quad \text{for every } k\in\Z ,
$$
where $r\in (0,\frac14 )$ is fixed as in the assumption. Let $C_r\geq 0$ be as in Lemma \ref{conditiongamma}.
We will show that
$$
\sup_{{\bf k}\in I, {\bf k}\not= \infty} \| f_{\bf k}\|_\infty \leq C_r .
$$
The proof goes by induction on $n$.

By Lemma \ref{conditiongamma}, for every $k\in\Z$, 
$$
\| \tilde{e}_k \circledast f\|_\infty \leq \frac{C_r}{h(d_k)} \leq C_r \, c_k 
$$
by the definition of $c_k$ and since $r\leq 1$, and therefore
$$
\| f_k\|_\infty \leq C_r \text{ for every } k\in\Z .
$$

Next, we assume that there exists $n\geq 1$ such that 
$$
\| f_{\bf k} \|_\infty \leq C_r \text{ for every } {\bf k}\in\Z^{n} .
$$

Let ${\bf k}=(k_\nu )_{1\leq \nu\leq n+1}\in\Z^{n+1}$. 

Assume first that there exist $1\leq \nu$, $\mu\leq n+1$ such that 
\begin{equation} \label{kmu1}
|k_\nu - k_\mu | > \frac{r}{4} (h(d_{k_\nu}) + h(d_{k_\mu}) ) .
\end{equation}
There exists $\tilde{{\bf k}}\in I$ such that 
$$
f_{\bf k} = \frac{\tilde{e}_{k_\nu} * \tilde{e}_{k_\mu}}{c_{k_\nu} \cdot c_{k_\mu}} \circledast f_{\tilde{{\bf k}}} , 
$$
and therefore, by the resolvent identity, by the induction hypothesis, by the definition of $c_{\bf k}$, and by \eqref{kmu1},
\begin{eqnarray*}
\| f_{\bf k}\|_\infty & = & \big\| \frac{\tilde{e}_{k_\nu} \circledast f_{\tilde{{\bf k}}} - \tilde{e}_{k_\mu} \circledast f_{\tilde{{\bf k}}}}{(k_\mu - k_\nu) c_{k_\nu} \cdot c_{k_\mu}} \big\|_\infty \\
& \leq & C_r \, \frac{1}{|k_\nu - k_\mu|} \, \big( \frac{1}{c_{k_\nu}} + \frac{1}{c_{k_\mu}} \big) \\
& = & C_r \, \frac{1}{|k_\nu - k_\mu|} \frac{r}{4} \big( h(d_{k_\nu}) + h (d_{k_\mu} ) \big) \\
& \leq & C_r .
\end{eqnarray*}

Hence, we may suppose that 
\begin{equation} \label{kmu2}
|k_\nu - k_\mu | \leq \frac{r}{4} (h(d_{k_\nu}) + h(d_{k_\mu}) ) \quad \text{for every } 1\leq \nu, \, \mu \leq n+1 .
\end{equation}
After a permutation of the indices, we may assume in addition that
$$
h( d_{k_1} ) = \max_{1\leq \nu\leq n+1} h (d_{k_\nu}) .
$$
From the estimate $|\lambda_{k_\nu} - \lambda_{k_1} | = |k_\nu - k_1 | \leq \frac{r}{2} h(d_{k_1} )$ we obtain
$$
\lambda_{k_\nu} \in B(\lambda_1 , \textstyle{\frac{3r}{4}} h( d_{k_1} ) ) \quad \text{for every } 1\leq \nu \leq n+1 .
$$
As a consequence, by Lemma \ref{l6}, and since $\lambda\mapsto e_\lambda \circledast f$ extends to a bounded analytic function on $B(\lambda_{k_1} , h(d_{k_1}) )$, 
\begin{equation} \label{ekf}
\tilde{e}_{\bf k} \circledast f = (e_{\lambda_{k_1}} * \, \dots \, * e_{\lambda_{k_{n+1}}} ) \circledast f = \frac{1}{2\pi i} \int_{\partial B(\lambda_{k_1} , \frac{3r}{4} \, h(d_{k_1}) )} \frac{e_\lambda \circledast f}{(\lambda - \lambda_{k_1}) \cdot \dots \cdot (\lambda - \lambda_{k_{n+1}})} d\lambda .
\end{equation}
Note that for every $\lambda \in \partial B(\lambda_1 , \frac{3r}{4}\, h(d_{k_1}) )$ and every $1\leq \nu\leq n+1$ one has
\begin{eqnarray*}
|\lambda - \lambda_{k_\nu} | & \geq & | \lambda - \lambda_{k_1} | - | \lambda_{k_1} - \lambda_{k_\nu} | \\
& \geq & \frac{3r}{4} \, h(d_{k_1}) - \frac{r}{4} (h(d_{k_1}) + h(d_{k_\nu}) ) \\
& \geq & \frac{r}{4} h(d_{k_1}) \\
& \geq & \frac{r}{4} h(d_{k_\nu} ) = \frac{1}{c_{k_\nu}} 
\end{eqnarray*}
or 
$$
\frac{1}{|\lambda - \lambda_{k_\nu} |} \leq c_{k_\nu} .
$$
This inequality, the equality \eqref{ekf}, and the decay condition from Lemma \ref{conditiongamma} yield 
\begin{eqnarray*}
\| \tilde{e}_{\bf k} \circledast f \|_\infty & \leq & \frac{3r}{4} h(d_{k_1} ) \, \sup_{\lambda\in \partial B(\lambda_{k_1} , \frac{3r}{4} \, h(d_{k_1}) )} \| e_\lambda \circledast f\|_\infty \,\, c_{k_1} \cdot \, \dots \, \cdot c_{k_{n+1}} \\
& \leq & C_r \,  c_{k_1} \cdot \, \dots \, \cdot c_{k_{n+1}} \\
& = & C_r \, c_{\bf k} .
\end{eqnarray*}
This implies 
$$
\| f_{\bf k}\|_\infty \leq C_r \quad \text{for every } {\bf k}\in\Z^{n+1} ,
$$
and by induction, the first claim is proved.

If the estimate \eqref{decayf} holds and if $c_k = \frac{4}{r} \, \frac{1}{d_k}$, then one may repeat the above proof replacing the function $h$ by the function $h(s) := s$.
\end{proof}

We are ready to prove Theorem \ref{main}.

\begin{proof}[Proof of Theorem \ref{main}]
Let $f\in L^\infty (\R_+ )$ and $\varphi \in C(\R_+ )$ be as in the hypothesis. Define the numbers $c_k >0$ as in Proposition \ref{l5} (depending on whether the condition \eqref{decayf} holds or not), and let the family $(f_{\bf k})_{{\bf k}\in I}$ be defined as in \eqref{fkdef}. By Proposition \ref{l5}, the family $(f_{\bf k})_{{\bf k}\in I}$ is uniformly bounded in $L^\infty (\R_+ )$.

Define the unit ball $B_M$, and the spaces $M$ and $M_*$ as above. We recall that the space $M$ embeds continuously into $L^\infty (\R_+ )$, and that by construction 
$$
M \subset \overline{\L1 \circledast f}^{(L^\infty , {\rm weak}^*)} .
$$ 

Let $T_*:\L1 \to \cL (M_*)$ be the algebra homomorphism defined in Lemma \ref{l3a}, and let 
$$
T : \L1 \to \cL (M) 
$$
be the algebra homomorphism given by $Tg (m) := g\circledast m$, $m\in M$. Clearly, $\| Tg \|_{\cL (M)} = \| T_* g \|_{\cL (M_* )}$ for every $g\in \L1$.

By Lemma \ref{l4}, and since $\sup_{{\bf k}\in I} \| f_{\bf k}\|_M \leq 1$ by Lemma \ref{l1} (c), for every $k\in\Z$,
\begin{eqnarray*}
\| T e_{\lambda_k} \|_{\cL (M)} = \| T_* e_{\lambda_k} \|_{\cL (M_*)} & = & \sup_{\bar{{\bf k}}\in I} \| e_{\lambda_k} \circledast f_{\bar{{\bf k}}} \|_M \\
& = & \sup_{\bar{{\bf k}}\in I} \| \tilde{e}_k \circledast \frac{\tilde{e}_{\bar{{\bf k}}} \circledast f }{c_{\bar{{\bf k}}}} \|_M \\
& = & \sup_{\bar{{\bf k}}\in I} \| \frac{\tilde{e}_{(\bar{{\bf k}},k)} \circledast f}{c_{(\bar{{\bf k}},k)}} \, c_k \|_M \\
& = & \sup_{\bar{{\bf k}}\in I} \| f_{(\bar{{\bf k}},k)} \, c_k \|_M \\
& \leq & c_k .
\end{eqnarray*}
By the definition of $c_k$ (see Proposition \ref{l5}), this leads to the estimate
\begin{equation} \label{Tek}
\| T e_{\lambda_k} \|_{\cL (M)} \leq \left\{ \begin{array}{ll}
C\, \frac{\log d_k}{d_k} & \text{or } \\[2mm]
C\, \frac{1}{d_k}  
\end{array} \right. \quad \text{for every } k\in\Z ,
\end{equation}
depending on whether the condition \eqref{decayf} holds or not.

By Lemma \ref{equivalenthomomorphism}, after replacing the space $M$ by a closed subspace $X$, if necessary, we can assume that the homomorphism $T$ is represented (in the strong sense) by a bounded $C_0$-semigroup $(T(t))_{t\geq 0} \in \cL (X)$. Since $T$ was defined by adjoint convolution, it follows that the semigroup $(T(t))_{t\geq 0}$ is the left-shift semigroup on $X$. If $A$ is the generator of this semigroup, then the estimate \eqref{Tek} implies 
$$
\| R ( \lambda_k , A) \|_{\cL (X)} \leq \left\{ \begin{array}{ll}
C\, \frac{\log d_k}{d_k} & \text{or } \\[2mm]
C\, \frac{1}{d_k}  
\end{array} \right. \quad \text{for every } k\in\Z ,
$$
depending on whether the condition \eqref{decayf} holds or not.

Now let $\beta\in\R$ be arbitrary, and let $k\in\Z$ be such that $|\beta -k|\leq 1$. By the resolvent identity and boundedness of the semigroup $(T(t))_{t\geq 0}$,
\begin{eqnarray*}
\| R(\lambda_\beta ,A) \|_{\cL (X)} & = & \| R(\lambda_k ,A) + i(k-\beta ) R(\lambda_\beta ,A) R(\lambda_k ,A) \|_{\cL (X)} \\
& \leq & (1+C) \, \| R(\lambda_k ,A) \|_{\cL (X)} \\
& \leq & \left\{ \begin{array}{ll}
C\, \frac{\log d_k}{d_k} & \text{or } \\[2mm]
C\, \frac{1}{d_k}  ,
\end{array} \right.
\end{eqnarray*}
depending on whether the condition \eqref{decayf} holds or not. 

By contractivity of the distance function we have
$$
| d_\beta -d_k | \leq |\beta -k| \leq 1 ,
$$
so that 
$$
d_\beta \leq d_k + 1 \leq 2 \, d_k ;
$$
recall that $d_k \geq 2$. This estimate for $d_\beta$ implies
$$
\frac{1}{d_k} \leq 2\, \frac{1}{d_\beta} \quad \text{ and } \quad \frac{\log d_k}{d_k} \leq 2\, \frac{\log d_\beta}{d_\beta} ,
$$
and therefore 
$$
\| R ( \lambda_\beta , A) \|_{\cL (X)} \leq \left\{ \begin{array}{ll}
C\, \frac{\log d_\beta}{d_\beta} & \text{or } \\[2mm]
C\, \frac{1}{d_\beta}  
\end{array} \right. \quad \text{for every } \beta\in\R ,
$$
depending on whether the condition \eqref{decayf} holds or not. 

It remains to show that $f\in {\rm range}\, T^*$. 
For every $g\in \L1$ we can estimate
\begin{eqnarray*}
\| Sg\|_{M^*} & = & \sup_{h\in B_{L^1}} \| S(g*h)\|_{M^*} \\
& = & \sup_{h\in B_{L^1}} \| g* Sh \|_{M^*} \\
& = & \| S\| \, \sup_{h\in B_{L^1}} \| g* \frac{Sh}{\| S\|} \|_{M^*} \\
& \leq & \| S\| \, \sup_{\stackrel{h\in L^1}{\| Sh\|_{M^*}\leq 1}} \| g * Sh \|_{M^*} \\
& = & \| S\| \, \sup_{m^* \in B_{M_*}} \| g* m^* \|_{M^*} \\
& = & \| S\| \, \| T_* g\|_{\cL (M_* )} \\
& = & \| S\| \, \| T g\|_{\cL (M )} .
\end{eqnarray*}
In other words, there is a bounded operator 
$$
R : \overline{{\rm range}\, T}^{\|\cdot \|_{\cL (M)}} \to M_* \subset M^*  
$$
such that $S = RT$. Hence, $T^*R^* = S^* : M^{**} \to L^\infty (\R_+ )$, which implies
\begin{equation} \label{rangest}
{\rm range}\, S^* \subset {\rm range}\, T^* .
\end{equation}
On the other hand, we recall from \eqref{ranges} that ${\rm range}\, S^* = M$. Since $f=f_\infty \in B_M$ by Lemma \ref{l1} (c), we thus obtain $f\in {\rm range}\, T^*$. 

Theorem \ref{main} is completely proved.
\end{proof}

\begin{remark}
It would be interesting to understand the geometric structure of the spaces $M$ and $X$, for example, whether they might be UMD spaces or spaces having nontrivial Fourier type.
\end{remark}

\section{The norm continuity problem} \label{sectionexamples}

In this section we present two examples showing that the norm continuity problem has a negative answer. In these two examples emphasis will be put on precise decay estimates for the resolvent along vertical lines. Before stating the two examples, we recall the following known result; see \cite[Theorem 4.9]{Pa83}, \cite[Theorem 4.1.3]{Fa83}. 

\begin{proposition} \label{logdecay}
Let $A$ be the generator of a bounded $C_0$-semigroup $(T(t))_{t\geq 0}$. Then the following are true:
\begin{itemize}
\item[(i)] If
$$
\| R(2+i\beta ,A) \| = o ( \frac{1}{\log |\beta |} ) \text{ as } |\beta |\to\infty ,
$$
then the semigroup $(T(t))_{t\geq 0}$ is immediately differentiable.
\item[(ii)] If
$$
\| R(2+i\beta ,A) \| = O ( \frac{1}{\log |\beta |} ) \text{ as } |\beta |\to\infty ,
$$
then the semigroup $(T(t))_{t\geq 0}$ is eventually differentiable.
\end{itemize}
\end{proposition}

The following is our first counterexample to the norm continuity problem. 
 
\begin{theorem} \label{example1}
There exists a Banach space $X$ and a uniformly bounded $C_0$-semigroup $(T(t))_{t\geq 0}\subset\cL (X)$ with generator $A$ such that:
\begin{itemize}
\item[(i)] the resolvent satisfies the estimate
$$
\| R(2+i\beta ,A) \| = O (\frac{1}{\log |\beta |} ) \quad \text{as } |\beta |\to\infty ,
$$
and in particular the resolvent satisfies the resolvent decay condition \eqref{res-cond},
\item[(ii)] $T(1) = 0$, that is, the semigroup $(T(t))_{t\geq 0}$ is nilpotent, and 
\item[(iii)] whenever $t_0\in [0,1)$, then the semigroup $(T(t))_{t\geq 0}$ is not norm-continuous for $t>t_0$.
\end{itemize}
\end{theorem}

\begin{proof}
Let $f=1_{[0,1]}$ be the characteristic function of the interval $[0,1]$. Since $f$ has compact support, the Laplace transform $\hat{f}$ and also the function $\lambda \mapsto e_\lambda\circledast f$ extend to entire functions, and for every $\lambda\in\C\setminus \{ 0\}$ and every $t\in\R_+$,
$$
(e_\lambda \circledast f )\, (t) = \left\{ \begin{array}{ll}
\frac{1}{\lambda} \, (1- e^{-\lambda (1-t)} ) & \text{if } 0\leq t\leq 1 , \\[2mm]
0 & \text{if } t>1 .
\end{array} \right.
$$
Hence, for every $\lambda\in\C\setminus \{ 0\}$,
\begin{equation} \label{basicest}
\| e_\lambda \circledast f\|_\infty \leq \left\{ \begin{array}{ll}
2 \, \frac{e^{-{\rm Re}\, \lambda}}{|\lambda |} & \text{if }{\rm Re}\, \lambda <0, \\[2mm]
2\, \frac{1}{|\lambda |} & \text{if } {\rm Re}\, \lambda \geq 0 , \, \lambda\not= 0 .
\end{array} \right.
\end{equation}
Let $\varphi \in C( \R_+ )$ be the function given by 
$$
\varphi (\beta ) = 1+\log^+ ( \beta  ) , \quad \beta \geq 0 ,
$$
where $\log^+$ is the positive part of the logarithm. Clearly $\lim_{\beta \to\infty } \varphi (\beta ) = +\infty$.

It follows from \eqref{basicest} that
$$
\sup_{\lambda\in\Sigma_\varphi} |\hat{f} (\lambda ) | < \infty ,
$$
so that $f$ satisfies the hypothesis of Theorem \ref{main}. By the definition of $\varphi$,
$$
\frac{3+ \log^+ |\beta |}{2} \leq 
d_\beta \leq 3+ \log^+ |\beta |  \quad \text{for every } \beta\in\R ,
$$ 
where, as before, we put $d_\beta := {\rm dist}\, (\lambda_\beta , \partial\Sigma_\varphi )$ and $\lambda_\beta = 2+i\beta$. From \eqref{basicest} one therefore also obtains for every $r\in (0,1)$ the estimate
\begin{eqnarray*}
\sup_{\lambda\in B(\lambda_\beta , r \, d_\beta )} \| e_\lambda \circledast f\|_\infty & \leq & 2\, \sup_{\lambda\in B(\lambda_\beta , r \, d_\beta )} \frac{\max \{ e^{-{\rm Re}\, \lambda} , 1\} }{|\lambda |} \\
& \leq & C \, \frac{e^{r \log |\beta |}}{|\beta | - r\, \log^+ |\beta |} \\
& \leq & C \, |\beta |^{-(1-r)} ,
\end{eqnarray*}
and in particular
$$
\sup_{\lambda\in B(\lambda_\beta , r \, d_\beta )} \| e_\lambda \circledast f\|_\infty \leq C_r \, \frac{1}{d_\beta} \quad \text{for every } \beta\in\R .
$$
This means that the function $f$ satisfies the decay condition \eqref{decayf} from Theorem \ref{main}. 

By Theorem \ref{main}, there exists a left-shift invariant Banach space $X\hookrightarrow L^\infty (\R_+ )$ such that the resolvent of the generator $A$ of the left-shift semigroup (which is strongly continuous on $X$) satisfies the decay estimate
$$
\| R(2+i\beta ,A) \|_{\cL (X)} \leq C \, \frac{1}{1+ \log^+ |\beta |} \quad \text{for every } \beta\in \R ,
$$
so that the resolvent satisfies the resolvent decay condition \eqref{res-cond}. Moreover, if $T:\L1 \to \cL (X)$ is the algebra homomorphism which is represented (in the strong sense) by the left-shift semigroup, then $f = 1_{[0,1]} \in {\rm range}\, T^*$. In particular, by Lemma \ref{necessary-cond} (d), the semigroup cannot be continuous for $t>t_0$ whenever $t_0\in [0,1)$. 

On the other hand, it follows from Theorem \ref{main} (iv) that every function in $X$ is supported in the interval $[0,1]$ so that the left-shift semigroup on $X$ vanishes for $t\geq 0$. 
\end{proof}

In the second example we show that there are also $C_0$-semigroups which are never norm-continuous, whose generator satisfies the resolvent decay condition \eqref{res-cond}, and the decay of the resolvent along vertical lines is even arbitrarily close to a logarithmic decay. Note that, by Proposition \ref{logdecay}, a logarithmic decay as in Theorem \ref{example1} (i) is not possible for semigroups which are not eventually norm-continuous, that is, norm-continuous for $t>t_0$.

\begin{theorem} \label{example2}
Let $h \in C (\R_+ )$ be a positive, increasing and unbounded function such that also the function $\log^+ / h$ is increasing and unbounded. Then there exists a Banach space $X$ and a uniformly bounded $C_0$-semigroup $(T(t))_{t\geq 0}\subset\cL (X)$ with generator $A$ such that:
\begin{itemize}
\item[(i)] the resolvent satisfies the estimate
$$
\| R(2+i\beta ,A) \| = O (\frac{h (|\beta | )}{\log |\beta |} ) \quad \text{as } |\beta |\to\infty ,
$$
and in particular the resolvent satisfies the resolvent decay condition \eqref{res-cond}, and
\item[(ii)] the semigroup $(T(t))_{t\geq 0}$ is not eventually norm-continuous.
\end{itemize}
\end{theorem}

\begin{proof}
Since the function $\log^+ /h$ is increasing and unbounded, then also the function $s\to s^{1/h(s)} = e^{\log s / h(s )}$ is increasing and unbounded for $s\geq 1$. We may assume that this function is strictly increasing for $s\geq 1$. In particular, the function $\psi$ given by
\begin{equation} \label{psicond}
\psi ( e\, s^{1 / h(s )} ) := \frac14 h(s) , \quad s\geq 1 ,
\end{equation}
is well-defined, increasing and unbounded.

Choose coefficients $a_n >0$ such that $1\geq a_n \geq a_{n+1}$ and such that 
\begin{equation} \label{entiregrowth}
\sum_{n=0}^\infty a_n r^{n+1} \leq r^{\psi (r)} \text{ for every } r\geq 1 ;
\end{equation}
it is an exercise to show that such coefficients exist (see also \cite[Problem 2, p.1]{Le96}).

We put
$$
f = \sum_{n=0}^\infty a_n 1_{[n,n+1]} .
$$
Then clearly $f\in L^\infty (\R_+ )$ and it follows from \eqref{entiregrowth} that the function $\lambda \mapsto e_\lambda \circledast f$ extends to an entire function. It is straightforward to show that for every $\lambda\in\C\setminus \{ 0\}$ and every $t\in\R_+$ one has
\begin{eqnarray*}
(e_\lambda \circledast f)\, (t) & = & \frac{e^{\lambda t}}{\lambda} (1-e^{-\lambda}) \sum_{n=[t]}^\infty a_{n} e^{-\lambda n} - a_{[t]} \frac{e^{-\lambda ([t]-t)}-1}{\lambda} \\
& = & \frac{e^{\lambda t}}{\lambda} (1-e^{-\lambda}) \sum_{n=[t]+1}^\infty a_{n} e^{-\lambda n} + a_{[t]} \frac{1-e^{-\lambda ([t]+1-t)}}{\lambda} .
\end{eqnarray*}
If ${\rm Re}\, \lambda < 0$, this yields the estimate
\begin{eqnarray*}
|(e_\lambda \circledast f)\, (t)| & \leq & \frac{2}{|\lambda |} \sum_{n=0}^\infty a_{n+[t]} e^{-{\rm Re}\, \lambda (n +1)} + \frac{2}{|\lambda |} \\
& \leq & \frac{2}{|\lambda |} \big( \sum_{n=0}^\infty a_{n} e^{-{\rm Re}\, \lambda (n +1)} + 1 \big) \\
& \leq & \frac{2}{|\lambda |} \big( e^{-{\rm Re}\, \lambda \psi ( e^{-{\rm Re}\, \lambda } )} + 2 \big) ,
\end{eqnarray*}
where in the second line we have used the fact that the sequence $(a_n)$ is decreasing, and in the third line we have used the estimate \eqref{entiregrowth}. If ${\rm Re}\, \lambda \geq 0$, then we obtain the estimate
$$
|(e_\lambda \circledast f)\, (t)| \leq \frac{C}{|\lambda |} 
$$
for some $C\geq 1$, so that
$$
\| e_\lambda \circledast f \|_\infty \leq \frac{C}{|\lambda |} \, \big( e^{-{\rm Re}\, \lambda \psi ( e^{-{\rm Re}\, \lambda } )} + 2 \big) \text{ for every } \lambda\in\C\setminus \{ 0\} .
$$

Let $\varphi (\beta ) := 1+\frac{\log^+ \beta}{h(\beta )}$, $\beta\geq 0$. By assumption, $\lim_{\beta \to\infty} \varphi (\beta ) = \infty$. Moreover, for every $\beta\in\R$ large enough,
\begin{eqnarray*}
\sup_{\lambda\in B(2+i\beta , 2+\varphi (\beta ))} \| e_\lambda \circledast f\|_\infty & \leq & \sup_{\lambda\in B(2+i\beta , 2+\varphi (\beta ))} \frac{C}{|\lambda |} \, \big( e^{-{\rm Re}\, \lambda \psi ( e^{-{\rm Re}\, \lambda } )} + 2 \big) \\
& \leq & C\, \frac{e^{\varphi (|\beta |) \psi (e^{\varphi (|\beta |)} )} +2}{|\beta | - \varphi (|\beta |)} .
\end{eqnarray*}
For all $\beta$ large enough we have $\varphi (|\beta | ) \leq \frac12 \, |\beta |$. Moreover, if $|\beta |\geq 2$, then
\begin{eqnarray*}
e^{\varphi (|\beta |) \psi (e^{\varphi (|\beta |)})} & = & e^{\psi (e\, |\beta |^{1/h(|\beta |)})} \, |\beta |^{\frac{\psi (e\, |\beta |^{1/h(|\beta |)})}{h( |\beta |)}} \\
& = & |\beta |^{\frac14 \, \frac{h(|\beta |)}{\log |\beta|}} \, |\beta |^\frac14 \\
& \leq & C\, |\beta |^\frac12 ,
\end{eqnarray*}
by definition of the functions $\varphi$ and $\psi$. Hence, if $\beta$ is large enough, then 
$$
\sup_{\lambda\in B(2+i\beta , 2+\varphi (\beta ))} \| e_\lambda \circledast f\|_\infty \leq C\, |\beta |^{-\frac12} .
$$

In particular,
$$
\sup_{\lambda\in\Sigma_\varphi} \| e_\lambda \circledast f\|_\infty < \infty .
$$
Moreover, if we let, as before, $\lambda_\beta = 2+i\beta$ and $d_\beta = {\rm dist}\, (\lambda_\beta , \partial\Sigma_\varphi )$, then 
$$
\frac{2+\varphi (|\beta | )}{2} \leq 
d_\beta \leq 2+ \varphi (|\beta | ) \text{ for every } \beta\in\R \text{ large enough} 
$$
and therefore
$$
\sup_{\lambda\in B(\lambda_\beta , d_\beta )} \| e_\lambda \circledast f\|_\infty \leq C\, \frac{1}{d_\beta} \quad  \text{for every } \beta\in\R . 
$$
By Theorem \ref{main}, there exists a left-shift invariant Banach space $X\hookrightarrow L^\infty (\R_+ )$ such that the left-shift semigroup $(T(t))_{t\geq 0}\subset\cL (X)$ is strongly continuous and such that the resolvent of the generator $A$ satisfies the estimate
$$
\| R(2+i\beta ,A) \| \leq C\, \frac{h(|\beta |)}{\log |\beta |} \quad \text{for every } \beta \in\R , \, |\beta |>1 , 
$$
so that the resolvent decay condition \eqref{res-cond} is satisfied. Moreover, still by Theorem \ref{main}, if $T: \L1 \to \cL (X)$ is the algebra homomorphism which is represented (in the strong sense) by the semigroup $(T(t))_{t\geq 0}$, then  $f\in {\rm range}\, T^*$. Since the function is not continuous on any interval of the form $(t_0 ,\infty )$, by Lemma \ref{necessary-cond} (d), the semigroup $(T(t))_{t\geq 0}$ cannot be eventually norm-continuous.
\end{proof}

{\it Acknowledgement.} The authors would like to thank the referee for his/her careful reading
and helpful remarks and suggestions to improve on the original version.

\nocite{HiPh57}
\nocite{ElMEn94}
\nocite{ElMEn96}
\nocite{GoWe99}
\nocite{DiUh77}
\nocite{MN91}
\nocite{BlMa96}
\nocite{Cj98}
\nocite{Il07}
\nocite{GGMS87}
\nocite{Cj99}
\nocite{Ki98}
\nocite{Gi90}
\nocite{Gi91}

\bibliographystyle{amsplain}
%\begin{thebibliography}{10}
%\bibliography{ralph}

\providecommand{\bysame}{\leavevmode\hbox to3em{\hrulefill}\thinspace}
\providecommand{\MR}{\relax\ifhmode\unskip\space\fi MR }
\providecommand{\MRhref}[2]{%
  \href{http://www.ams.org/mathscinet-getitem?mr=#1}{#2}
}
\providecommand{\href}[2]{#2}

\end{document}